\documentclass[11pt]{amsart}
\usepackage{amsmath}\usepackage{amscd}\usepackage{amssymb}\usepackage{amsfonts}
\newtheorem{theorem}{Theorem}[section]
\newtheorem{lemma}[theorem]{Lemma}

\newtheorem{proposition}[theorem]{Proposition}

\theoremstyle{definition}

\theoremstyle{remark}
\newtheorem{remark}[theorem]{Remark}
\numberwithin{equation}{section}

\begin{document}
\title[Liouville property for $f$-harmonic functions]
{Liouville property for $f$-harmonic functions with polynomial growth}
\author{Jia-Yong Wu}
\address{Department of Mathematics, Shanghai Maritime University, 1550 Haigang Avenue, Shanghai 201306, P. R. China}\email{jywu81@yahoo.com}

\date{Received: March 25, 2018; Accepted: August 29, 2018.}
\begin{abstract}
We prove a Liouville property for any $f$-harmonic function with
polynomial growth on a complete noncompact smooth metric measure space
$(M,g,e^{-f}dv)$ when the Bakry-\'Emery Ricci curvature is nonnegative
and its diameter of geodesic sphere has sublinear growth.
\end{abstract}
\maketitle

\section{Introduction}\label{Int1}
In this paper, we will study a Liouville property for any $f$-harmonic function with
polynomial growth on complete smooth metric measure spaces with the nonnegative
Bakry-\'Emery Ricci curvature.

Recall that an $n$-dimensional complete smooth metric measure space, customarily denoted by $(M,g,e^{-f}dv_g)$, is an $n$-dimensional Riemannian manifold $(M,g)$ together with a weighted volume element $e^{-f}dv_g$ for some $f\in C^\infty(M)$, and
the volume element $dv_g$ induced by the Riemannian metric $g$. The $f$-Laplacian $\Delta_f$ on smooth metric measure space $(M,g,e^{-f}dv)$, self-adjoint with respect to $e^{-f}dv_g$, is defined by
\[
\Delta_f=\Delta-\nabla f\cdot\nabla.
\]
On smooth metric measure space $(M, g, e^{-f}dv_g)$, a smooth function $u$ is called weighted harmonic
(or $f$-harmonic) if $\Delta_f u=0$, and $f$-subharmonic if $\Delta_fu\geq 0$.
It is easy to see that the absolute value of an $f$-harmonic function is a nonnegative
$f$-subharmonic function. If $f$ is constant, all notions mentioned above reduce to the classical Riemannian case.

A natural generalization of Ricci curvature associated to smooth metric measure space
$(M,g,e^{-f}dv_g)$ is called $m$-Bakry-\'Emery Ricci curvature, which is defined by
\[
\mathrm{Ric}_f^m=\mathrm{Ric}+\nabla^2 f-\frac{1}{m}df\otimes df
\]
for some number $m>0$. Bakry and \'Emery \cite{[BE]} extensively studied this curvature
tensor and its relationship to diffusion processes. When $m$ is finite, the Bochner
formula for the $m$-Bakry-\'Emery Ricci curvature can be read as
\[
\frac 12\Delta_f|\nabla u|^2\geq\frac{(\Delta_f u)^2}{m+n}+\langle\nabla\Delta_f u, \nabla
u\rangle+\mathrm{Ric}_f^m(\nabla u, \nabla u),
\]
which could be viewed as the Bochner formula for the Ricci curvature of an
$(n+m)$-dimensional manifold. Hence many classical results for manifolds
with Ricci curvature bounded below can be extended to smooth metric
measure spaces with $m$-Bakry-\'Emery Ricci curvature bounded below, see for example
\cite{[BQ1],[LD], [Wang], [WW0], [Wu0], [WuWu3]} for details. When $m=\infty$, one denotes
\[
\mathrm{Ric}_f=\mathrm{Ric}_f^\infty=\mathrm{Ric}+\nabla^2 f.
\]
In particular, if
\[
\mathrm{Ric}_f=\lambda\, g
\]
for some $\lambda\in\mathbb{R}$, then $(M, g, e^{-f}dv_g)$ is a gradient Ricci soliton,
which arises from the singularity analysis of the Ricci flow \cite{[Ham]}. When
$\mathrm{Ric}_f$ is bounded below, some analytical and geometric properties on smooth
metric measure spaces were also possibly explored as long as the weight function
$f$ is assumed to be some condition. We refer to the work of Dung \cite{[Du]}, Lott
\cite{[Lott1]}, Wei and Wylie \cite{[WW]}, and \cite{[Wu],[WuWu1],[WuWu2], [WuWu3]}
for further details.

Recently, many Liouville theorems for $f$-harmonic functions on smooth metric measure spaces have been studied. For example,
Wei and Wylie \cite{[WW]} proved that any positive $f$-harmonic function with some growth condition with $\mathrm{Ric}_f\geq K>0$ must be constant. Brighton \cite{[Bri]} generalized Yau's gradient result \cite{[Yau2]} and derived a gradient estimate for positive $f$-harmonic functions. He also obtained a Liouville theorem for positive bounded $f$-harmonic functions. Inspired by Brighton's argument, Munteanu and Wang \cite{[MuWa]} applied the De Giorgi-Nash-Moser theory to get a sharp gradient estimate and a Liouville-type result for positive $f$-harmonic functions on smooth metric measure spaces. Dung and Dat \cite{[DD]} proved that any positive weighted $p$-eigenfunction is constant if it is of sublinear growth on smooth metric measure spaces with $\mathrm{Ric}^m_f\geq0$. Recently, Wang et al. \cite{[WZZZ]} proved positive weighted $p$-eigenfunction is constant when $\mathrm{Ric}^m_f$ is bounded below by using the Moser's iterative technique.

Motivated by the classical works of \cite{[Li],[Li-Sch], [Yau]}, many $L_f^p$-Liouville ($0<p<\infty$) theorems on smooth metric measure spaces were extensively studied. Here $L_f^p$-Liouville theorem means that every $L_f^p$-integrable $f$-harmonic function on smooth metric measure spaces is constant. Recall that $u$ is called $L_f^p$-integrable, i.e. $u\in L_f^p$, if its $L_f^p$-norm,
defined as
\[
\|u\|_p:=\left(\int_M |u|^pe^{-f}dv\right)^{1/p},
\]
is finite. In \cite{[LD]} Li proved an $L_f^1$-Liouville theorem for $f$-subharmonic functions on smooth metric measure spaces with $\mathrm{Ric}_f^m\geq-c(1+r(x)^2)$, where $r(x):=d(o,x)$ is the distance function starting from a base point $o\in M$,
$0<m<\infty$. Pigola, Rimoldi, and Setti \cite{[PRS]} proved a sharp $L_f^p$-Liouville theorem with $p>1$ for $f$-subharmonic functions on smooth metric measure spaces without any curvature assumption. When $f$ is bounded, the author \cite{[Wu]} derived some $L_f^p$-Liouville theorems with $0<p\leq 1$ for $f$-subharmonic functions on smooth metric measure spaces with some Bakry-Emery curvature assumptions. When $f$ is not bounded,  P. Wu and the author \cite{[WuWu1], [WuWu2], [WuWu3]} systematically studied various $L_f^p$-Liouville theorems with $0<p\leq 1$ on smooth metric measure spaces under different Bakry-Emery curvature conditions. In particular, we obtained a sharp $L_f^1$-Liouville theorem when $\mathrm{Ric}_f\geq 0$ by using $f$-heat kernel Gaussian upper estimates on smooth metric measure spaces (see also \cite{[CL]}).

\vspace{0.4cm}

In this paper, we prove a new Liouville theorem for any $f$-harmonic function with polynomial growth
when the Bakry-\'Emery Ricci curvature is nonnegative and its diameter of geodesic sphere is sublinear.
The proof uses a local $f$-Neumann Poincar\'e inequality, and an iterated argument combined with a
cut off function technique, using a similar argument of \cite{[Ca]} (see also \cite{[WuWu1],[WuWu2]}).
\begin{theorem}\label{Mainthm}
Let $(M,g,e^{-f}dv)$ be an $n$-dimensional complete noncompact smooth metric measure space
with
\[
\mathrm{Ric}_f\ge 0\quad\mathrm{and}\quad  \sup_{x\in M}|f(x)|<+\infty.
\]
For a base point $o\in M$ and $R>0$, if the diameter of geodesic sphere $B_o(R)$
has a sublinear growth:
\[
\mathrm{diam}\, \partial B_o(R):=\sup_{x,y\in\partial B_o(R)}d(x,y)=o(R),\quad R\to \infty,
\]
then any $f$-harmonic function with polynomial growth is constant.
\end{theorem}

\begin{remark}
Munteanu and Wang \cite{[MuWa]} proved the Liouville property for all \emph{positive} $f$-harmonic
functions on smooth metric measure spaces. Our Liouville theorem hold for $f$-harmonic
functions which are not necessarily positive.
\end{remark}

\begin{remark}
In the proof of  Theorem  \ref{Mainthm}, the finiteness of function $f$ makes
the iterated procedure to work efficiently. It is interesting to ask whether
this condition can be removed.
\end{remark}

\begin{remark}
For smooth Riemannian manifolds whose sectional curvature satisfies a quadratic
decay lower bound and whose geodesic spheres have sublinear growth, A. Kasue
\cite{[Ka1], [Ka2]} proved that any harmonic function with polynomial growth is
constant (see also G. Carron \cite{[Ca]} for further generalizations). In some
sense, Theorem \ref{Mainthm} generalizes their results to the setting of smooth
metric measure spaces.
\end{remark}

This rest of the paper is organized as follows. In Section \ref{sec2}, we recall
some known results, such as $f$-volume comparison theorems, local Poincar\'e
inequalities for the Bakry-\'Emery Ricci curvature bounded below. In Section
\ref{sec3}, we adapt an iterated argument to prove Theorem \ref{Mainthm}.

\textbf{Acknowledgement}.
This work is supported by the NSFC (11671141) and the Natural Science Foundation of Shanghai (17ZR1412800).


\section{Preliminaries}\label{sec2}
Let $(M,g,e^{-f}dv)$ be an $n$-dimensional complete noncompact smooth metric measure space.
For any point $x\in M$ and $R>0$, we denote
\[
A(R):=A(x,R)=\sup_{y\in B_x(3R)}|f(y)|.
\]
We often write $A$ for short. In \cite{[WW]} Wei and Wylie
proved a relative $f$-volume comparison theorem.
\begin{lemma}\label{comp}
Let $(M,g,e^{-f}dv)$ be an $n$-dimensional complete noncompact
smooth metric measure space. If
\[
\mathrm{Ric}_f\geq-(n-1)\, K
\]
for some constant $K\geq0$, then
\begin{equation}\label{volcomp}
\frac{V_f(B_x(R_1,R_2))}{V_f(B_x(r_1,r_2))}\leq
\frac{V^{n+4A}_K(B_x(R_1,R_2))}{V^{n+4A}_K(B_x(r_1,r_2))}
\end{equation}
for any $0<r_1<r_2,\ 0<R_1<R_2$, $r_1\leq R_1,\ r_2\leq R_2$, where
$B_x(R_1,R_2)=B_x(R_2)\backslash B_x(R_1)$, and
$A=A(x,\frac{1}{3}R_2)$. Here ${V^{n+4A}_K(B_x(r))}$ denotes the
volume of the ball in the model space $M^{n+4A}_K$, i.e., the simply
connected space form of dimension $n+4A$ with constant sectional
curvature $-K$.
\end{lemma}

For model space $V^{n+4A}_K(B_x(r))$, if $K>0$, this model space is
the hyperbolic space; if $K=0$, this model space is the Euclidean
space. In any case, we have the following lower and upper estimates
\begin{equation}\label{props1}
\omega_{n+4A}\cdot r^{n+4A}\leq V_K^{n+4A}(B_x(r))\leq \omega_{n+4A}\cdot
r^{n+4A}e^{(n-1+4A)\sqrt{K}\,r},
\end{equation}
where $\omega_{n+4A}$ is the volume of the unit
ball in $(n+4A)$-dimensional Euclidean space.

\

In \cite{[WuWu2]}, P. Wu and the author proved a local Neumann
Poincar\'e inequality on complete smooth metric measure spaces.
\begin{lemma}\label{NeuPoin}
Let $(M,g,e^{-f}dv)$ be an $n$-dimensional complete noncompact
smooth metric measure space with
\[
\mathrm{Ric}_f\geq-(n-1)\, K
\]
for some constant $K\geq0$. Fix a point $o\in M$ and $R>0$. For any
$x\in B(o,R)$, we have
\begin{equation}\label{Nepoinineq}
\int_{B_x(r)}|\varphi-\varphi_{B_x(r)}|^2 e^{-f}dv\leq
c_1e^{c_2A+c_3(1+A)\sqrt{K}r}\cdot r^2\int_{B_x(r)}|\nabla
\varphi|^2e^{-f}dv
\end{equation}
for all $0<r<R$ and $\varphi\in C^\infty(B_x(r))$, where
$\varphi_{B_x(r)}:=\int_{B_x(r)}\varphi d\mu/\int_{B_x(r)}d\mu$;
$A=\sup_{y\in B_o(3R)}|f(y)|$, and where the coefficient $c_i$ are
constants, depending only on the dimension $n$.
\end{lemma}
We remark that the local $f$-Neumann Poincar\'e inequality and so called $f$-volume
doubling property can imply a local Sobolev inequality and a mean value inequality for
solutions to the $f$-heat equation on smooth metric measure spaces. Putting these
inequalities together, we can furthermore prove Moser's Harnack inequality by using
the De Giorgi-Nash-Moser theory. For detailed discussions, the interested
readers can refer to \cite{[WuWu2]}.

\

In this paper, we are concerned with some analytic and geometric results
on $(M,g,e^{-f}dv)$ when $\mathrm{Ric}_f\ge0$. Under this curvature
assumption, Lemma \ref{comp} implies that
\begin{proposition}\label{volbound}
Let $(M,g,e^{-f}dv)$ be an $n$-dimensional complete noncompact
smooth metric measure space. Fix a point $o\in M$. If
$\mathrm{Ric}_f\ge0$, then
\[
V_f(B_o(R))\leq C \,R^{n+4A}
\]
for all $R>1$, where $C$ is a constant, which only depends on $n$,
$V_f(B_o(1))$ and $A:=\sup_{y\in B_o(R)}|f(y)|$.
\end{proposition}
\begin{proof}
In Lemma \ref{comp}, we let $r_1=R_1=0$, $r_2=1$, $R_2=R$ and $x=o$, and then
\begin{equation*}
\begin{aligned}
\frac{V_f(B_o(R))}{V_f(B_o(1))}\leq \frac{V^{n+4A}_0(B_o(R))}{V^{n+4A}_0(B_o(1))}
&\leq C(n,A)R^{n+4A},
\end{aligned}
\end{equation*}
where $A=\sup_{y\in B_o(R)}|f(y)|$. Note that we also used \eqref{props1}
in the above second inequality. Hence
\[
V_f(B_o(R))\leq C \,R^{n+4A}
\]
for all $R>1$, where $C$ is a positive constant depending on $n$, $A$ and
$V_f(B_o(1))$.
\end{proof}

\vspace{0.4cm}

In \cite{[GriSal]}, A. Grigor'yan and L. Saloff-Coste introduced a useful notion:
\emph{remote ball}. Recall that, for a fixed point $o\in M$, a ball $B(x,\rho)$ in $M$ is
called \emph{remote ball} if
\[
d(o,x)\geq 3\rho.
\]
When $\mathrm{Ric}_f\geq -c\,r^{-2}(x)$, A. Grigor'yan and L. Saloff-Coste
\cite{[GriSal]} showed that all remote balls in manifold $M$ satisfy
the volume doubling property, the Poincar\'e inequality and the parabolic Harnack
inequality. In our case, for all remote balls $B(x,\rho)$, Lemma \ref{NeuPoin}
can be restated as
\begin{proposition}\label{NeuPoin2}
Let $(M,g,e^{-f}dv)$ be an $n$-dimensional complete noncompact
smooth metric measure space with
\[
\mathrm{Ric}_f\ge0.
\]
Let $o\in M$ be a fixed point and $R>0$. For all remote balls $B_x(\rho)$, $x\in B_o(R)$,
we have
\begin{equation}\label{Nepoinineq}
\int_{B_x(\rho)}|\varphi-\varphi_{B_x(\rho)}|^2e^{-f}dv\leq
c_4 e^{c_5\bar{A}}\cdot R^2\int_{B_x(\rho)}|\nabla\varphi|^2e^{-f}dv
\end{equation}
for any $\varphi\in C^\infty(B_x(\rho))$, where
$\varphi_{B_x(\rho)}:=\int_{B_x(\rho)}\varphi e^{-f}dv/\int_{B_x(\rho)}e^{-f}dv$ and
$\bar{A}=\sup_{y\in B_o(3R)}|f(y)|$. Here $c_4$ and $c_5$ are constants
depending only on $n$.
\end{proposition}


\section{Proof of Theorem \ref{Mainthm}}\label{sec3}
In this section, we first follow the argument of \cite{[Ca]} to give a useful
gradient estimate in the integral sense. The proof depends on the cut-off function
technique and iterated procedure. Then we apply this estimate to prove Theorem
\ref{Mainthm}.
\begin{lemma}\label{keylem}
Let $(M,g,e^{-f}dv)$ be an $n$-dimensional complete noncompact
smooth metric measure space. For any a point $o\in M$ and $R>0$,
let $u$ be a $f$-harmonic smooth function on $B_o(2R)$, and
\[
\mathrm{diam}\, \partial B_o(r)\leq \epsilon r\quad \mathrm{with}\quad \epsilon\in (0, 1/12)
\]
for all $r\in [R,2R]$. Assume that all remote balls $B_x(\rho)$, $x\in B_o(2R)$, satisfies
a local Poincar\'e inequality:
\begin{equation}\label{Nepoinineq}
\int_{B_x(\rho)}|\varphi-\varphi_{B_x(\rho)}|^2e^{-f}dv\leq
c_4 e^{c_5\bar{A}}\cdot R^2\int_{B_x(\rho)}|\nabla\varphi|^2e^{-f}dv
\end{equation}
for any $\varphi\in C^\infty(B_x(\rho))$, where
$\varphi_{B_x(\rho)}:=\int_{B_x(\rho)}\varphi e^{-f}dv/\int_{B_x(\rho)}e^{-f}dv$ and
$\bar{A}:=\sup_{y\in B_o(6R)}|f(y)|$. Then we have
\begin{equation}\label{gradineq}
\int_{B_o(R)}|\nabla u|^2 e^{-f}dv\leq \delta^{\frac{1}{\epsilon}-1}\int_{B_o(2R)}|\nabla u|^2 e^{-f}dv,
\end{equation}
where $\delta:=\left(\frac{9c_4 e^{c_5\bar{A}}}{1+9c_4 e^{c_5\bar{A}}}\right)^{\frac 16}$,
and $\bar{A}:=\sup_{y\in B_o(6R)}|f(y)|$.
\end{lemma}

\begin{proof}
For any a point $o\in M$ and $R>0$, suppose that $u\colon B_o(2R)\to\mathbb{R}$
be a $f$-harmonic function and  $c\in \mathbb{R}$ is a real number. Choosing
\[
r\in [R+3\epsilon R ,2R-3\epsilon R],
\]
where $\epsilon\in (0, 1/12)$, the diameter
hypothesis in lemma implies that there exists some $x\in \partial B_o(r)$ such that
\[
B_o(r+\epsilon R)\setminus B_o(r)\subset B_x(\epsilon R+\epsilon r).
\]
Let $\eta(x)$ be a $C^2$ cut-off function with support in $B_o(r+\epsilon R)$. That is,
\begin{equation*}
\eta(x)=\left\{ \begin{aligned}
&1&&\mathrm{if}\,\,x\in B_o(r),\\
&\frac{r+\epsilon R-d(o,x)}{\epsilon R}
&&\mathrm{if}\,\,x\in B_o(r+\epsilon R)\setminus B_o(r),\\
&0&&\mathrm{if}\,\,x\in M \setminus B_o(r+\epsilon R).\\
\end{aligned}\right.
\end{equation*}

We easily observe that
\[
|\nabla(\eta(u-c))|^2=\eta^2 |\nabla(u-c)|^2+2\eta(u-c)\langle\nabla\eta,\nabla(u-c)\rangle+(u-c)^2|\nabla\eta|^2.
\]
Integrating this identity with respect to the weighted
measure $e^{-f}dv$, we have
\[
\int_M |\nabla(\eta(u-c))|^2=\int_M \eta^2 |\nabla(u-c)|^2+2\eta(u-c)\langle\nabla\eta,\nabla(u-c)\rangle+(u-c)^2|\nabla\eta|^2.
\]
Notice that
\begin{equation}
\begin{split}\label{intepro}
\int_M \eta^2 |\nabla(u-c)|^2+2\eta(u-c)\langle \nabla\eta, \nabla(u-c)\rangle&=\int_M\langle \nabla((u-c)\eta^2), \nabla(u-c)\rangle\\
&=-\int_M(u-c)\eta^2\Delta_f(u-c)\\
&=0,
\end{split}
\end{equation}
where we have used integration by part with respect to $e^{-f}dv$ and the fact
that $u$ is $f$-harmonic.
Therefore,
\[
\int_M |\nabla(\eta(u-c))|^2e^{-f}dv=\int_{B(o,r+\epsilon R)}(u-c)^2 |\nabla\eta|^2e^{-f}dv.
\]
According to the definition of $\eta(x)$, by the above equality, we get that
\begin{equation*}
\begin{split}
\int_{B_o(r)}|\nabla u|^2e^{-f}dv&\le \int_{B_o(r+\epsilon R)} |\nabla(\eta(u-c))|^2e^{-f}dv\\
&=\int_{B_o(r+\epsilon R)}(u-c)^2|d\eta|^2e^{-f}dv\\
&\le\frac{1}{\epsilon^2 R^2}\int_{B_o(r+\epsilon R)\setminus B_o(r)}(u-c)^2e^{-f}dv\\
&\le\frac{1}{\epsilon^2 R^2}\int_{B_x(\epsilon R+\epsilon r)}(u-c)^2e^{-f}dv.
\end{split}
\end{equation*}
For the right hand side of the above inequality, since $\epsilon \le 1/12$,
this implies that the ball $B(x,\epsilon R+\epsilon r)$ is remote.
Hence if choosing
\[
c=u_{B_x(\epsilon (R+r))}=\frac{1}{\int_{B_x(\epsilon (R+r))}e^{-f}dv}\cdot\int_{B_x(\epsilon (R+r))}ue^{-f}dv,
\]
then using the local $f$-Poincar\'e inequality \eqref{Nepoinineq} and the fact that $r+R\le 3 R$,
we have that
\[
\int_{B_o(r)}|\nabla u|^2e^{-f}dv\le 9c_4 e^{c_5\bar{A}}\int_{B_x(3\epsilon R)}|\nabla u|^2e^{-f}dv,
\]
where $\bar{A}=\sup_{y\in B_o(6 R)}|f(y)|$. Also noticing that
\[
B_x(3\epsilon R)\subset B_o(r+3\epsilon R)\setminus B_o(r-3\epsilon R),
\]
hence we get
\begin{equation*}
\begin{split}
\int_{B_o(r-3\epsilon R)}|\nabla u|^2e^{-f}dv&\le\int_{B_o(r)}|\nabla u|^2e^{-f}dv\\
&\le9c_4e^{c_5\bar{A}}\int_{B_o(r+3\epsilon R)\setminus B_o(r-3\epsilon R)}|\nabla u|^2e^{-f}dv.
\end{split}
\end{equation*}
That is, for all $r\in [R,R-6\epsilon R]$, we have
\[
\int_{B_o(r)}|\nabla u|^2e^{-f}dv
\le\frac{9c_4e^{c_5\bar{A}}}{1+9c_4e^{c_5\bar{A}}}\int_{B_o(r+6\epsilon R)}|\nabla u|^2e^{-f}dv.
\]
Iterating this inequality, we finally get
\[
\int_{B_o(R)}|\nabla u|^2e^{-f}dv
\le\left(\frac{9c_4e^{c_5\bar{A}}}{1+9c_4 e^{c_5\bar{A}}}\right)^N\int_{B_o(2R)}|\nabla u|^2e^{-f}dv
\]
provided that
\[
N 6\epsilon R\le R.
\]
At last the desired result follows by choosing
$\delta=\left(\frac{9c_4 e^{c_5\bar{A}}}{1+9c_4 e^{c_5\bar{A}}}\right)^{\frac 16}$.
\end{proof}

\

Now we are ready to apply a similar argument of \cite{[Ca]} to
prove Theorem \ref{Mainthm} by using Lemma \ref{keylem}.

\begin{proof}[Proof of Theorem \ref{Mainthm}]
Let $u\colon M\rightarrow \mathbb{R}$ be a $f$-harmonic function with polynomial growth
of order $\nu$, namely,
\[
|u(x)|\le C (1+d(o,x))^\nu .
\]
For $R>>1$, we define
\[
I_R:=\int_{B_o(R)} |\nabla u|^2 e^{-f}dv
\]
and
\[
\epsilon(r):=\sup_{t\ge r} \frac{\rho(t)}{t},
\]
where $\displaystyle \rho(t):=\sup_{x, y\in \partial B_o(t)} d(x,y)$.

To estimate $I_R$, we shall introduce a $C_0^2(M)$ cut off function $\xi(x)$, which satisfies
\[
\xi(x)=\begin{cases}
1&\text{in}\quad B_o(R),\\
\frac{2R-d(o,x)}{R}&\text{in}\quad B_o(2R)\setminus B_o(R),\\
0&\text{in}\quad M\setminus B_o(2R).
\end{cases}
\]
Hence we have
\begin{equation}
\begin{split}\label{cappo}
I_R&\le\int_{B_o(2R)}|\nabla(\xi u)|^2e^{-f}dv\\
&=\int_{B_o(2R)}\Big[|u|^2|\nabla\xi|^2+\xi^2|\nabla u|^2+2\xi u\langle\nabla\xi,\nabla u\rangle\Big]e^{-f}dv\\
&=\int_{B_o(2R)}|u|^2|\nabla\xi|^2e^{-f}dv\\
&\le C R^{2\nu+n+4\bar{A}-2},
\end{split}
\end{equation}
where $C$ depends on $n$, $\bar{A}:=\sup_{x\in M}|f(x)|$ and $V_f(B_o(1))$.
Here in the third line of \eqref{cappo}, we used the same proposition as \eqref{intepro};
in the last line of \eqref{cappo}, we used Proposition \ref{volbound}.

On the other hand, if we iterate the inequality \eqref{gradineq} proved in Lemma \ref{keylem},
we can show that for all $R$ such that $\epsilon(R)\le 1/12$ :
\[
I_R\le\delta^{-\ell+\sum_{j=0}^{\ell-1} \frac{1}{\epsilon(2^j R)}}\, I_{2^\ell R},
\]
where $\delta=\left(\frac{9c_4 e^{c_5\bar{A}}}{1+9c_4 e^{c_5\bar{A}}}\right)^{\frac 16}$.
Applying \eqref{cappo} to the right hand side of the above inequality yields
\begin{equation}\label{eniequ}
I_R\le C\cdot \exp\left\{\ell
\Bigg[\left(\frac{1}{\ell}\sum_{j=0}^{\ell-1}\frac{1}{\epsilon(2^jR)}-1\right)\cdot\ln\delta+(2\nu+n+4\bar{A}-2)\cdot\ln 2\Bigg]\right\},
\end{equation}
where in the above inequality, $\bar{A}:=\sup_{x\in M}|f(x)|$,
$\delta:=\left(\frac{9c_4 e^{c_5\bar{A}}}{1+9c_4 e^{c_5\bar{A}}}\right)^{\frac 16}$
and $C:=C\left(n,\nu,R,\bar{A},V_f(B_o(1))\right)$.

For all sufficiently large $R$, the Cesaro convergence theorem shows that
\[
\lim_{\ell\to +\infty} \frac 1\ell\sum_{j=0}^{\ell-1} \frac{1}{\epsilon\left(2^j R\right)}=+\infty,
\]
since $\epsilon(2^j R)\to 0$ when $j\to+\infty$. Meanwhile, by our assumption:
\[
\sup_{x\in M}|f(x)|<+\infty,
\]
we know that $\bar{A}<+\infty$,
\[
0<\delta<1\quad \mathrm{and}\quad \ln\delta<0
\]
for any $R$ (even though $\ell\to +\infty$). Therefore, if $\ell\to +\infty$ in \eqref{eniequ},
then we conclude that
\[
I_R=0
\]
for all sufficiently large $R$. Therefore $u$ is constant.
\end{proof}


\bibliographystyle{amsplain}

\end{document}